\documentclass{amsart}
\usepackage{amssymb}
\usepackage{amsmath}
\usepackage{amsfonts}
\usepackage{hyperref}
\usepackage{yhmath}
\usepackage{framed}
\newtheorem{theorem}{Theorem}[section]
\newtheorem{lemma}[theorem]{Lemma}
\newtheorem{corollary}[theorem]{Corollary}

\newtheorem{prop}[theorem]{Proposition}

\newtheorem{remark}[theorem]{Remark}

\newcommand{\R}{\mathbb{R}}

\begin{document}
\title[Improved GWP for defocusing sixth-order Boussinesq equations]{Improved global well-posedness for defocusing sixth-order Boussinesq equations}

\author{Dan-Andrei Geba and Evan Witz}

\address{Department of Mathematics, University of Rochester, Rochester, NY 14627, U.S.A.}
\email{dangeba@math.rochester.edu}
\address{Department of Mathematics, University of Rochester, Rochester, NY 14627, U.S.A.}
\email{ewitz@ur.rochester.edu}
\date{}

\begin{abstract}
This article studies the global well-posedness (GWP) for a class of defocusing, generalized sixth-order Boussinesq equations, extending a previous result obtained by Wang-Esfahani \cite{WE-14} for the case when the nonlinear term is cubic.
\end{abstract}


\maketitle

\section{Introduction}

\subsection{Background of the problem}
Our goal is to study the initial value problem (IVP) associated to generalized sixth-order Boussinesq equations given by
\begin{equation}
\begin{cases}
u_{tt}-u_{xx}-\beta u_{xxxx}-u_{xxxxxx}\,=\,(f(u))_{xx}, \quad u=u(t,x)\in \mathbb{R}_+\times \R \to \R,\\
u(0,x)\,=\,g(x),\qquad u_t(0,x)\,=\,h_x(x),\\
\end{cases}
\label{main}
\end{equation}
where $\beta=\pm 1$. This type of equations is physically relevant, being originally derived by Christov-Maugin-Velarde \cite{CMV-96} in the context of shallow fluid layers and nonlinear atomic chains. It was also later tied to modeling small amplitude and long capillary-gravity waves by Daripa-Hua \cite{DH-99}, along with describing nonlinear dynamics in elastic crystals by Maugin \cite{M-99}.
  
The IVP \eqref{main} with power-type nonlinearity (i.e., $f(u)\simeq u^p$) has received considerable interest lately, with a focus on local and global existence of solutions, as well as on sufficient conditions for blow-up in finite time. Esfahani-Farah \cite{EF-12} proved first that \eqref{main} with $f(u)=u^2$ is locally well-posed (LWP) for $(g,h)\in H^s(\R)\times H^{s-1}(\R)$ when $s>-1/2$, a result which was improved by Esfahani-Wang \cite{EW-14} to allow $s>-3/4$. For the case when $f(u)=|u|^\alpha u$ with $\alpha>0$,  Esfahani-Farah-Wang \cite{EFW-12} showed that \eqref{main} is LWP when $h=\tilde{h}_x$ and either $(g,\tilde{h})\in H^1(\R)\times L^2(\R)$ or $(g,\tilde{h})\in L^2(\R)\times \dot{H}^{-1}(\R)$ (this under the further restriction $\alpha<4$). The same paper also established small data GWP in the case when $f(u)=-|u|^\alpha u$ with $\alpha>0$, $h=\tilde{h}_x$, and $(g,\tilde{h})\in H^2(\R)\times \dot{H}^{1}(\R)$, and derived sufficient conditions for blow-up phenomena. Lastly, Wang-Esfahani \cite{WE-14} demonstrated that \eqref{main} with $f(u)=|u|^2 u$ is GWP for $(g,h)\in H^s(\R)\times H^{s-2}(\R)$ when $3/2<s<2$. This literature parallels the progress made on similar issues for the classical generalized Boussinesq equation
\begin{equation*}
u_{tt}-u_{xx}+u_{xxxx}\,=\,(f(u))_{xx},
\end{equation*}
by Linares \cite{L93}, Fang-Grillakis \cite{FG96}, Farah \cite{F092,F09}, Farah-Linares \cite{FL10}, Kishimoto-Tsugawa \cite{KT10}, Farah-Wang \cite{FW12}, and Kishimoto \cite{K13}.

\subsection{Main result and outline of the paper}\label{main-result}
Our aim here is to generalize the result obtained by Wang-Esfahani to the class of IVP \eqref{main} with $f(u)=|u|^{2k} u$, where $k\geq 2$ is an integer. The following is the main contribution of this article.

\begin{theorem} 
The Cauchy problem \eqref{main}, with $f(u)=|u|^{2k} u$ and $k\geq 2$ being an integer, is GWP for $(g,h)\in H^s(\R)\times H^{s-2}(\R)$ and $2-2/(3k)<s<2$. In addition, the solution $u$ satisfies
\begin{equation}
\sup_{0\leq t\leq T} \left\{\|u(t)\|^2_{H^s(\R)}+\|(-\Delta)^{-1/2}u_t(t)\|^2_{H^{s-2}(\R)}\right\}\,\lesssim\,(1+T)^{\frac{4-2s}{6ks-12k+4}+}
\label{u-hs}
\end{equation}
for all $T>0$, where the implicit constant depends strictly on $s$, $\|g\|_{H^s(\R)}$, and $\|h\|_{H^{s-2}(\R)}$.  
\label{main-th} 
\end{theorem}

\noindent To comment on this theorem, let us start by observing the formal conservation of the energy
\begin{equation}
\aligned
E(u)(t):=\, &\frac{1}{2}\|u_{xx}(t)\|^2_{L^2(\R)}-\frac{\beta}{2}\|u_x(t)\|^2_{L^2(\R)}+\frac{1}{2}\|u(t)\|^2_{L^2(\R)}\\&+\frac{1}{2}\|(-\Delta)^{-1/2}u_t(t)\|^2_{L^{2}(\R)}+\frac{1}{2k+2}\|u(t)\|^{2k+2}_{L^{2k+2}(\R)},
\endaligned
\label{energy}
\end{equation}
which also satisfies\footnote{The energy is nonnegative and this is why we associate a defocusing terminology to this equation.}
\begin{equation}
E(u)(t)\simeq \, \|u(t)\|^2_{H^2(\R)}+\|(-\Delta)^{-1/2}u_t(t)\|^2_{L^{2}(\R)}+\|u(t)\|^{2k+2}_{L^{2k+2}(\R)},
\label{energy-bd}
\end{equation}
even for $-2<\beta<2$, due to the well-known inequality
\begin{equation*}
\|v_x\|^2_{L^2(\R)}\leq \|v_{xx}\|_{L^2(\R)}\|v\|_{L^2(\R)}.
\end{equation*}
This conservation partly motivates the challenging nature of our result, since the energy can be infinite and, thus, impractical for certain data $(g,h)\in H^s(\R)\times H^{s-2}(\R)$ with $s<2$. To deal with this shortcoming, we rely on the \textit{I-method}, also known as the \textit{method of almost conservation laws}, pioneered by Colliander-Keel-Staffilani-Takaoka-Tao \cite{CKSTT-01, CKSTT-01-2} for KdV and nonlinear Schr\"{o}dinger equations, respectively. However, the implementation of this technique is slightly less direct here, as Boussinesq equations are not scale-invariant, unlike the dispersive equations for which the method was originally designed. A final observation is that, by comparison to Wang-Esfahani's work \cite{WE-14} (i.e., $k=1$), we obtain an improved key multilinear estimate \eqref{multi-I-int} which enhances the predicted range $2-1/(2k)<s<2$ to the one proven in the above theorem. Furthermore, the proof of this bound is streamlined to include fewer cases than its counterpart in \cite{WE-14}.

The structure of this paper is as follows. In section \ref{toolbox}, we introduce the analytic toolbox, which includes the functional spaces and the appropriate estimates to be used in the analysis, along with the smoothing operator $I$ and its properties. In section \ref{LWP-th}, we work on proving a LWP result for the equation obtained by the application of operator $I$ to the original Boussinesq equation \eqref{main}.  We follow this in section \ref{multilinear} with the proof of the crucial multilinear estimate, which allows us to demonstrate Theorem \ref{main-th} in the final section.
 
\subsection*{Acknowledgements}
The first author was supported in part by a grant from the Simons Foundation $\#\, 359727$.


\section{Analytic toolbox} \label{toolbox}

\subsection{Notational conventions}

First, we agree to write $A\lesssim B$ when $A\leq CB$ and $A\ll B$ when $A\leq C^{-1}B$, where $C>2$ is a constant depending only upon parameters which are considered fixed throughout the paper. Moreover, we write $A\sim B$ to denote that both $A\lesssim B$ and $B\lesssim A$ are valid. We also use the notation $a\pm=a\pm\varepsilon$ when $0<\varepsilon\ll 1$ is a universal constant.

Secondly, as is the custom for $w=w(t,x): I\times \R\to\R$ with $I\subseteq \R$ being an arbitrary time interval, we rely on
\begin{equation*}
\aligned
\|w\|_{L_t^pL^q_x(I\times \R)}:&= \left(\int_I\|w(t,\cdot)\|_{L^q(\R)}^{p}dt\right)^{1/p},\\
\|w\|_{L_x^pL^q_t(I\times \R)}:&= \left(\int_\R\|w(\cdot,x)\|_{L^q(I)}^{p}dt\right)^{1/p},
\endaligned
\end{equation*}
with the obvious modification when $p=\infty$. Furthermore, for ease of notation, we write
\begin{equation*}
\aligned
\|w\|_{L_t^pL_x^q}=\|w\|_{L_t^pL^q_x(\R\times \R)},&\qquad \|w\|_{L_x^pL_t^q}=\|w\|_{L_x^pL^q_t(\R\times \R)},\\
\|w\|_{L_{t\in[0,\delta]}^pL_x^q}=\|w\|_{L_t^pL^q_x([0,\delta]\times \R)},&\qquad \|w\|_{L_x^pL_{t\in[0,\delta]}^q}=\|w\|_{L_x^pL^q_t([0,\delta]\times \R)}.
\endaligned
\end{equation*}
When $p=q$, we simplify the notation and write $L_t^pL_x^p=L_x^pL_t^p=L_{t,x}^p$.

Finally, we denote by
\begin{equation*}
\widehat{v}(\xi):=\int_\R e^{-ix\xi}\,v(x)\,dx\quad \text{and}\quad \widetilde{w}(\tau, \xi) :=\int_{\R^2} e^{-i(t\tau+x\xi)}\,w(t,x)\,dt\,dx
\end{equation*}
the Fourier transform of $v=v(x)$ and the spacetime Fourier transform of $w=w(t,x)$, respectively.


\subsection{Relevant norms and related estimates}

We start by writing $\langle a\rangle:=(1+a^2)^{1/2}$ and $\omega(\xi):=(\xi^2-\beta\xi^4+\xi^6)^{1/2}$, which allows us to to define the Sobolev and Bourgain-type norms
\begin{equation*}
\aligned
&\|v\|_{H^s}:=\|\langle \xi\rangle^s\widehat{v}(\xi)\|_{L^2_\xi(\R)},\\ 
\|w\|_{X^{s,\theta}}&:=\|\langle \xi\rangle^s\langle |\tau|-\omega(\xi)\rangle^\theta\widetilde{w}(\tau,\xi)\|_{L^2_{\tau,\xi}(\R^2)},
\endaligned
\end{equation*}
for $s$, $\theta\in \R$. Working directly with these norms, one can easily prove the classical bound
\begin{equation}
\|w\|_{L^\infty_tH^s_x}\lesssim \|w\|_{X^{s,\theta}}
\label{Sob-X}
\end{equation}
and the inclusion $X^{s,\theta}\subset C(\R; H^s(\R))$, both for all $s\in \R$ and $\theta>1/2$.

For $\delta>0$, we also use the truncated norm
\begin{equation}
\|z\|_{X^{s,\theta}_\delta}:=\inf_{w=z\, \text{on}\, [0,\delta]} \|w\|_{X^{s,\theta}}.
\label{xstd}\end{equation}
We observe that according to Remark 3.1 in \cite{EF-12} one has 
\begin{equation}
\|w\|_{X^{s,\theta}} \simeq \|\langle \xi\rangle^s\langle |\tau|-|\xi|^3+\frac{\beta}{2}|\xi|\rangle^\theta\widetilde{w}(\tau,\xi)\|_{L^2_{\tau,\xi}(\R^2)},
\label{xst-equiv}
\end{equation}
which suggests that we may derive estimates for this norm using known bounds for the Airy equation $v_t+v_{xxx}=0$. Indeed, we can prove this next result.

\begin{prop}
The following estimates hold true:
\begin{align}
\|w\|_{L_t^pL_x^q}\lesssim \|w\|&_{X^{0,\frac{1}{2}+}} ,\quad \frac{3}{p}+\frac{1}{q}=\frac {1}{2}, \quad 2\leq q\leq \infty, \,\quad (\text{Strichartz})\label{w-Str}\\
&\|w_x\|_{L_x^\infty L_t^2}\lesssim \|w\|_{X^{0,\frac{1}{2}+}}, \qquad\qquad\qquad (\text{Kato smoothing})\label{w-K}\\
\|w&\|_{L_x^4L_t^\infty}\lesssim \|D^{1/4}_xw\|_{X^{0,\frac{1}{2}+}}, \quad\qquad\qquad\ \ (\text{maximal function}) \label{w-mf}
\end{align}
where $D_x=(-\Delta)^{1/2}$ is the multiplier operator given by $\widehat{D_xv}(\xi)=|\xi|\widehat{v}(\xi)$. The same bounds are valid with $L_t^pL_x^q$, $L_x^pL^q_t$, and $X^{s,\theta}$ replaced by $L_{t\in[0,\delta]}^pL_x^q$, $L_x^pL_{t\in[0,\delta]}^q$, and $X^{s,\theta}_\delta$, respectively, for all $\delta>0$.

\end{prop}

\begin{proof}
First, we record the estimates proven by Kenig-Ponce-Vega (Lemma 2.4 in \cite{KPV-89}, Theorems 3.5 and 3.7 in \cite{KPV-93}) for solutions to the Airy equation:
\begin{align*}
\|v\|_{L_t^pL_x^q}\lesssim \|v&(0)\|_{L^2} ,\quad \frac{3}{p}+\frac{1}{q}=\frac {1}{2}, \quad 2\leq q\leq \infty, \\
&\|v_x\|_{L_x^\infty L_t^2}\lesssim \|v(0)\|_{L^2}, \\
\|v&\|_{L_x^4L_t^\infty}\lesssim \|D^{1/4}_xv(0)\|_{L^2}.
\end{align*}
It is easy to see that if $z_t-z_{xxx}=0$ then $v(t,x)=z(-t,x)$ solves the Airy equation and, hence, the previous three 
bounds also hold true for $z$. Then, we can use standard arguments (e.g, Lemma 2.9 in Tao \cite{T-06}) to transform these estimates into ones involving Bourgain-type norms:
\begin{align*}
\|v\|_{L_t^pL_x^q}\lesssim \|v\|&_{X^{0,\frac{1}{2}+}_{\tau=\pm \xi^3}} ,\quad \frac{3}{p}+\frac{1}{q}=\frac {1}{2}, \quad 2\leq q\leq \infty,\\
&\|v_x\|_{L_x^\infty L_t^2}\lesssim \|v\|_{X^{0,\frac{1}{2}+}_{\tau=\pm \xi^3}}, \\
\|v&\|_{L_x^4L_t^\infty}\lesssim \|D^{1/4}_xv\|_{X^{0,\frac{1}{2}+}_{\tau=\pm \xi^3}},
\end{align*}
with
\[
\|v\|_{X^{s,\theta}_{\tau=\pm \xi^3}} := \|\langle \xi\rangle^s\langle \tau\mp\xi^3\rangle^\theta\widetilde{v}(\tau,\xi)\|_{L^2_{\tau,\xi}(\R^2)}.
\]
Following this, a direct calculation shows that if $v(t,x)= w(t, x\pm \frac{\beta}{2}t)$ then $\widetilde{v}(\tau,\xi)= \widetilde{w}(\tau\mp \frac{\beta}{2}\xi,\xi)$ and, consequently,
\[
\|v\|_{X^{s,\theta}_{\tau=\pm \xi^3}}=\|w\|_{X^{s,\theta}_{\tau=\pm \xi^3\mp\frac{\beta}{2}\xi}}:= \|\langle \xi\rangle^s\langle \tau\mp\xi^3\pm\frac{\beta}{2}\xi\rangle^\theta\widetilde{w}(\tau,\xi)\|_{L^2_{\tau,\xi}(\R^2)}.
\] 
Furthermore, we infer based on \eqref{xst-equiv} that
\begin{equation*}
\|w\|_{X^{s,\theta}} \simeq \|w_1\|_{X^{s,\theta}_{\tau=\xi^3-\frac{\beta}{2}\xi}}+\|w_2\|_{X^{s,\theta}_{\tau=- \xi^3+\frac{\beta}{2}\xi}},
\end{equation*}
where
\[
w=w_1+w_2, \qquad \widetilde{w_1}=\widetilde{w_1}\cdot {\bf 1}_{\{\tau\xi\geq 0\}},  \qquad \widetilde{w_2}=\widetilde{w_1}\cdot {\bf 1}_{\{\tau\xi<0\}}.
\]
It is then clear that \eqref{w-Str}-\eqref{w-mf} follow as the combined result of the mathematical facts developed so far in this proof. For the same estimates, but in which one restricts the domain of variable $t$, we can use \eqref{xstd} to deduce
\[
\|w\|_{X^{0,\frac{1}{2}+}_\delta}\gtrsim \inf_{w=z\, \text{on}\, [0,\delta]} \|z\|_{L_t^pL_x^q}\geq \|w\|_{L_{t\in [0,\delta]}^pL_x^q}
\]
and a similar approach works for the other two bounds.
\end{proof}

\begin{remark}
In addition to these inequalities, we will also employ in our analysis
\begin{equation}
\|w\|_{L^\infty_{t,x}}\lesssim \|w\|_{X^{\frac{1}{2}+,\frac{1}{2}+}}
\label{w-infty}
\end{equation}
and its localized in time version, which is the joint conclusion of Sobolev embeddings and \eqref{w-Str} with $(p,q)=(\infty, 2)$.
\end{remark}

As an application of this proposition, we derive the following multilinear estimate.

\begin{corollary}
If $k\geq 2$ and $s\geq 1/2-2/(2k+1)$, then the inequality 
\begin{equation}
\|w_1\cdot w_2\ldots \cdot w_{2k+1}\|_{X^{s-1,0}}\lesssim \|w_1\|_{X^{s,\frac{1}{2}+}}\ldots \|w_{2k+1}\|_{X^{s,\frac{1}{2}+}}
\label{prod-xst}
\end{equation}
is valid.
\end{corollary}     

\begin{proof}
It is easy to see that for all $s$ we have
\[
\langle\xi_1+\ldots+\xi_n\rangle^{s-1}\lesssim 1+\langle\xi_1\rangle^{s-1}+\ldots+\langle\xi_n\rangle^{s-1},
\]
which implies the Leibniz-type bound
\[
\aligned
\|w_1\cdot w_2\ldots \cdot w_{2k+1}\|_{X^{s-1,0}}\lesssim\, &\|w_1\cdot w_2\ldots \cdot w_{2k+1}\|_{L^2_{t,x}}+\|J^{s-1}w_1\cdot w_2\ldots \cdot w_{2k+1}\|_{L^2_{t,x}}\\
&+\ldots+\|w_1\cdot\ldots w_{2k}\cdot J^{s-1}w_{2k+1}\|_{L^2_{t,x}},
\endaligned
\]
where $J$ is the multiplier operator given by $\widehat{Jv}(\xi)=\langle\xi\rangle \widehat{v}(\xi)$. This effectively reduces the proof of the desired bound to the ones of
\[
 \|w_1\cdot w_2\ldots \cdot w_{2k+1}\|_{L^2_{t,x}}\lesssim \|w_1\|_{X^{s,\frac{1}{2}+}}\ldots \|w_{2k+1}\|_{X^{s,\frac{1}{2}+}}
\]
and
\[
 \|J^{s-1}w_1\cdot w_2\ldots \cdot w_{2k+1}\|_{L^2_{t,x}}\lesssim \|w_1\|_{X^{s,\frac{1}{2}+}}\ldots \|w_{2k+1}\|_{X^{s,\frac{1}{2}+}}.
\]
However, these follows by using H\"{o}lder's inequality, Sobolev embeddings, \eqref{w-Str}, and \eqref{w-infty}:
\[
\aligned
\|w_1\cdot w_2\ldots \cdot w_{2k+1}\|_{L^2_{t,x}}&\lesssim \|w_1\|_{L^{4k+2}_{t,x}}\ldots\|w_{2k+1}\|_{L^{4k+2}_{t,x}}\\
&\lesssim\|J^{\frac{1}{2}-\frac{2}{2k+1}}w_1\|_{L^{4k+2}_{t}L^{\frac{2k+1}{k-1}}_{x}}\ldots\|J^{\frac{1}{2}-\frac{2}{2k+1}}w_{2k+1}\|_{L^{4k+2}_{t}L^{\frac{2k+1}{k-1}}_{x}}\\
&\lesssim \|J^{\frac{1}{2}-\frac{2}{2k+1}}w_1\|_{X^{0,\frac{1}{2}+}}\ldots\|J^{\frac{1}{2}-\frac{2}{2k+1}}w_{2k+1}\|_{X^{0,\frac{1}{2}+}}\\
&\lesssim \|w_1\|_{X^{s,\frac{1}{2}+}}\ldots\|w_{2k+1}\|_{X^{s,\frac{1}{2}+}}
\endaligned
\]
and
\[
\aligned
\|J^{s-1}w_1\cdot w_2\ldots &\cdot w_{2k+1}\|_{L^2_{t,x}}\\&\lesssim \|J^{s-1}w_1\|_{L^\infty_{t,x}}\|w_2\|_{L^{4k}_{t,x}}\ldots\|w_{2k+1}\|_{L^{4k}_{t,x}}\\
&\lesssim\|J^{s-1}w_1\|_{L^\infty_{t,x}}\|J^{\frac{1}{2}-\frac{1}{k}}w_2\|_{L^{4k}_{t}L^{\frac{4k}{2k-3}}_{x}}\ldots\|J^{\frac{1}{2}-\frac{1}{k}}w_{2k+1}\|_{L^{4k}_{t}L^{\frac{4k}{2k-3}}_{x}}\\
&\lesssim \|J^{s-1}w_1\|_{X^{\frac{1}{2}+,\frac{1}{2}+}}\|J^{\frac{1}{2}-\frac{1}{k}}w_2\|_{X^{0,\frac{1}{2}+}}\ldots\|J^{\frac{1}{2}-\frac{1}{k}}w_{2k+1}\|_{X^{0,\frac{1}{2}+}}\\
&\lesssim \|w_1\|_{X^{s,\frac{1}{2}+}}\|w_2\|_{X^{s,\frac{1}{2}+}}\ldots\|w_{2k+1}\|_{X^{s,\frac{1}{2}+}}.
\endaligned
\]
\end{proof}


\subsection{Estimates for the linear equation}

Here, we revisit bounds satisfied by solutions to the linear equation
\begin{equation}
w_{tt}-w_{xx}-\beta w_{xxxx}-w_{xxxxxx}\,=\,F, \quad w(0)\,=\,g,\quad w_t(0)\,=\,h_x,
\label{lin-eq}
\end{equation}
claimed in \cite{EF-12} to be derived in a similar way with the corresponding estimates for the classical linear Boussinesq equation
\[
w_{tt}-w_{xx}+w_{xxxx}\,=\,F,
\]
proven in \cite{F09}. In order to state them, we need to introduce the cutoff function $\eta=\eta(t)\in C^\infty_0(\R)$ satisfying $0\leq \eta\leq 1$ and 
\[
\eta(t)=
\begin{cases}
1, \quad \text{if}\ |t|\leq 1,\\
0, \quad \text{if}\ |t|\geq 2,\\
\end{cases}
\]
and we also let $\eta_\delta(t):=\eta(t/\delta)$ for $0<\delta$. For purposes of completeness, we present arguments with full details for these estimates.

First, we address the homogeneous equation (i.e., \eqref {lin-eq} with $F\equiv 0$).

\begin{prop}
For the IVP
\begin{equation*}
w_{tt}-w_{xx}-\beta w_{xxxx}-w_{xxxxxx}\,=\,0, \quad w(0)\,=\,g,\quad w_t(0)\,=\,h_x,
\end{equation*}
we have that
\begin{equation}
\|\eta w\|_{X^{\sigma,\theta}}+\|\eta (-\Delta)^{-\frac 12}w_t\|_{X^{\sigma-2,\theta}}\lesssim \|g\|_{H^\sigma}+\|h\|_{H^{\sigma-2}}
\label{lin-hom}
\end{equation}
holds true for all $\sigma$, $\theta\in \R$, with the implicit constant depending solely on $\eta$, $\sigma$, and $\theta$.
\end{prop}
\begin{proof}
The proof follows the blueprint of the one for Lemma 2.1 in \cite{F09} and we emphasize here the main steps. First, direct computations using the Fourier transform yield
\begin{equation*}
\widetilde{\eta w}(\tau,\xi)= \frac{\widehat{\eta}(\tau-\omega(\xi))}{2}\left(\widehat{g}(\xi)+\frac{\xi\widehat{h}(\xi)}{\omega(\xi)}\right)+\frac{\widehat{\eta}(\tau+\omega(\xi))}{2}\left(\widehat{g}(\xi)-\frac{\xi\widehat{h}(\xi)}{\omega(\xi)}\right)
\end{equation*}
and
\begin{equation*}
\aligned
\widetilde{\eta (-\Delta)^{-\frac 12}w_t}(\tau,\xi)=\, &\frac{\widehat{\eta}(\tau+\omega(\xi))}{2i}\left(\frac{\omega(\xi)\widehat{g}(\xi)}{|\xi|}-\frac{\xi\widehat{h}(\xi)}{|\xi|}\right)\\&-\frac{\widehat{\eta}(\tau-\omega(\xi))}{2i}\left(\frac{\omega(\xi)\widehat{g}(\xi)}{|\xi|}+\frac{\xi\widehat{h}(\xi)}{|\xi|}\right).
\endaligned
\end{equation*}
Next, if we rely on the definition of the $X^{\sigma,\theta}$ norm, the fact that $\eta\in C^\infty_0$, $\omega(\xi)\geq 0$, and 
\begin{equation}
||\tau|-\omega(\xi)|\leq\min\{|\tau-\omega(\xi)|, |\tau+\omega(\xi)|\},
\label{mod}
\end{equation}
then we deduce
\begin{equation*}
\|\eta w\|_{X^{\sigma,\theta}}\lesssim \|\langle \xi\rangle^\sigma\widehat{g}(\xi)\|_{L^2_\xi}+\left\|\langle \xi\rangle^\sigma\frac{\xi\widehat{h}(\xi)}{\omega(\xi)}\right\|_{L^2_\xi}
\end{equation*}
and
\begin{equation*}
\|\eta (-\Delta)^{-\frac 12}w_t\|_{X^{\sigma-2,\theta}}\lesssim \left\|\langle \xi\rangle^{\sigma-2}\frac{\omega(\xi)\widehat{g}(\xi)}{|\xi|}\right\|_{L^2_\xi}+\left\|\langle \xi\rangle^{\sigma-2}\widehat{h}(\xi)\right\|_{L^2_\xi}.
\end{equation*}
Finally, if we take into account the easily-derived approximation
\begin{equation}
\omega(\xi)\simeq |\xi|\langle\xi\rangle^{2},
\label{symbol}
\end{equation}
we reach the desired conclusion.
\end{proof}
\begin{remark}
The corresponding estimate written in \cite{EF-12} has on the right-hand side the larger norm $\|h\|_{H^{\sigma-1}}$, instead of $\|h\|_{H^{\sigma-2}}$.
\end{remark}

The second result of this subsection concerns the inhomogeneous equation with zero data (i.e., \eqref{lin-eq} with $g=h\equiv 0$).

\begin{prop}
For the IVP
\begin{equation*}
w_{tt}-w_{xx}-\beta w_{xxxx}-w_{xxxxxx}\,=\,F, \quad w(0)\,=\,w_t(0)\,=\,0,
\end{equation*}
the estimate
\begin{equation}
\|\eta_\delta w\|_{X^{\sigma,\theta_1}}+\|\eta_\delta (-\Delta)^{-\frac 12}w_t\|_{X^{\sigma-2,\theta_1}}\lesssim \delta^{1-\theta_1+\theta_2}\|(-\Delta)^{-\frac 12}F\|_{X^{\sigma-2,\theta_2}}
\label{lin-inhom}
\end{equation}
is valid for all $0<\delta\leq 1$, $\sigma\in\R$, and $-1/2<\theta_2\leq 0\leq \theta_1\leq \theta_2+1$, with the implicit constant depending solely on $\eta$, $\delta$, $\sigma$, $\theta_1$, and $\theta_2$.
\end{prop}
\begin{proof}
The argument is similar in structure to the one for Lemma 2.2 in \cite{F09} and starts by working with Duhamel's formula to derive
\begin{equation*}\aligned
&\eta_\delta(t) \widehat{w}(t,\xi)=e^{it\omega(\xi)}\widehat{U_+}(t,\xi)-e^{-it\omega(\xi)}\widehat{U_-}(t,\xi),\\
\eta_\delta(t) &\widehat{(-\Delta)^{-\frac 12}w_t}(t,\xi)=e^{it\omega(\xi)}\widehat{V_+}(t,\xi)+e^{-it\omega(\xi)}\widehat{V_-}(t,\xi),
\endaligned
\end{equation*}
where $U_\pm=U_\pm(t,x)$ and $V_\pm=V_\pm(t,x)$ are defined through their spatial Fourier transform according to
\begin{equation*}\aligned
\widehat{U_\pm}(t,\xi)=\eta_\delta(t) \int_0^t e^{\mp it'\omega(\xi)}\frac{\widehat{F}(t',\xi)}{2i\omega(\xi)}\,dt',\qquad\widehat{V_\pm}(t,\xi)=\eta_\delta(t) \int_0^t e^{\mp it'\omega(\xi)}\frac{\widehat{F}(t',\xi)}{2|\xi|}\,dt'.
\endaligned
\end{equation*}
Next, on the basis of the definition of the $X^{\sigma,\theta}$ norm, $\theta_1\geq 0$, $\omega(\xi)\geq 0$, and 
\begin{equation*}
\max\{||\tau+\omega(\xi)|-\omega(\xi)|, ||\tau-\omega(\xi)|-\omega(\xi)|\}\leq |\tau|,
\end{equation*}
we infer that 
\begin{equation*}
\aligned
\|\eta_\delta w\|_{X^{\sigma,\theta_1}}&\lesssim \|\langle \xi\rangle^\sigma\langle \tau\rangle^{\theta_1}\widetilde{U_+}(\tau,\xi)\|_{L^2_{\tau,\xi}}+\|\langle \xi\rangle^\sigma\langle \tau\rangle^{\theta_1}\widetilde{U_-}(\tau,\xi)\|_{L^2_{\tau,\xi}}\\
&\lesssim \|\langle \xi\rangle^\sigma\|\widehat{U_+}(\cdot,\xi)\|_{H^{\theta_1}_t}\|_{L^2_{\xi}}+\|\langle \xi\rangle^\sigma\|\widehat{U_-}(\cdot,\xi)\|_{H^{\theta_1}_t}\|_{L^2_{\xi}}
\endaligned
\end{equation*}
and
\begin{equation*}
\aligned
\|\eta_\delta (-\Delta)^{-\frac 12}w_t\|_{X^{\sigma-2,\theta_1}}&\lesssim \|\langle \xi\rangle^{\sigma-2}\langle \tau\rangle^{\theta_1}\widetilde{V_+}(\tau,\xi)\|_{L^2_{\tau,\xi}}+\|\langle \xi\rangle^{\sigma-2}\langle \tau\rangle^{\theta_1}\widetilde{V_-}(\tau,\xi)\|_{L^2_{\tau,\xi}}\\
&\lesssim \|\langle \xi\rangle^{\sigma-2}\|\widehat{V_+}(\cdot,\xi)\|_{H^{\theta_1}_t}\|_{L^2_{\xi}}+\|\langle \xi\rangle^{\sigma-2}\|\widehat{V_-}(\cdot,\xi)\|_{H^{\theta_1}_t}\|_{L^2_{\xi}}.
\endaligned
\end{equation*}
Following this, we deal with the inner Sobolev norms above by applying an estimate also used in the proof of Lemma 2.2 in \cite{F09}, which takes the form
\begin{equation*}
\left\| \eta_\delta(t) \int_0^t f(t')\,dt'\right\|_{H^{\theta_1}_t}\lesssim  \delta^{1-\theta_1+\theta_2} \|f\|_{H^{\theta_2}_t},
\end{equation*}
with $\delta$, $\theta_1$, and $\theta_2$ satisfying the hypothesis of our proposition. Thus, we obtain
\begin{equation*}
\aligned
\|\eta&_\delta  w\|_{X^{\sigma,\theta_1}}\\
&\lesssim  \delta^{1-\theta_1+\theta_2}\left\{\left\|\langle \xi\rangle^\sigma\left\|e^{- it\omega(\xi)}\frac{\widehat{F}(t,\xi)}{\omega(\xi)}\right\|_{H^{\theta_2}_t}\right\|_{L^2_{\xi}}+\left\|\langle \xi\rangle^\sigma\left\|e^{it\omega(\xi)}\frac{\widehat{F}(t,\xi)}{\omega(\xi)}\right\|_{H^{\theta_2}_t}\right\|_{L^2_{\xi}}\right\}\\
&\simeq  \delta^{1-\theta_1+\theta_2}\left\{\left\|\frac{\langle\xi\rangle^\sigma}{\omega(\xi)}\langle\tau-\omega(\xi)\rangle^{\theta_2}\widetilde{F}(\tau,\xi)\right\|_{L^2_{\tau,\xi}}+\left\|\frac{\langle\xi\rangle^\sigma}{\omega(\xi)}\langle\tau+\omega(\xi)\rangle^{\theta_2}\widetilde{F}(\tau,\xi)\right\|_{L^2_{\tau,\xi}}\right\}
\endaligned
\end{equation*}
and
\begin{equation*}
\aligned
\|&\eta_\delta (-\Delta)^{-\frac 12}w_t\|_{X^{\sigma-2,\theta_1}}\\
&\lesssim  \delta^{1-\theta_1+\theta_2}\left\{\left\|\langle \xi\rangle^{\sigma-2}\left\|e^{- it\omega(\xi)}\frac{\widehat{F}(t,\xi)}{|\xi|}\right\|_{H^{\theta_2}_t}\right\|_{L^2_{\xi}}+\left\|\langle \xi\rangle^{\sigma-2}\left\|e^{it\omega(\xi)}\frac{\widehat{F}(t,\xi)}{|\xi|}\right\|_{H^{\theta_2}_t}\right\|_{L^2_{\xi}}\right\}\\
&\simeq  \delta^{1-\theta_1+\theta_2}\left\{\left\|\frac{\langle\xi\rangle^{\sigma-2}}{|\xi|}\langle\tau-\omega(\xi)\rangle^{\theta_2}\widetilde{F}(\tau,\xi)\right\|_{L^2_{\tau,\xi}}+\left\|\frac{\langle\xi\rangle^{\sigma-2}}{|\xi|}\langle\tau+\omega(\xi)\rangle^{\theta_2}\widetilde{F}(\tau,\xi)\right\|_{L^2_{\tau,\xi}}\right\}.
\endaligned
\end{equation*}
The argument is concluded by taking advantage of \eqref{mod}, $\theta_2\leq 0$, and \eqref{symbol}.
\end{proof}


\subsection{Basic elements of the \textit{I-method}}

We follow the exposition in Colliander-Keel-Staffilani-Takaoka-Tao \cite{CKSTT-04} and introduce the smooth, even Fourier multiplier $m:\R\to\R^+$ given by
\begin{equation}
m(\xi)=\begin{cases}
1, \quad &\text{if}\ |\xi|\leq 1,\\
|\xi|^{-1}, \quad &\text{if}\ |\xi|\geq 2,
\end{cases}
\label{m-def}
\end{equation}
which permits us to define the family of multiplier operators $(I^\sigma_N)_{\sigma\geq 0, N\geq 1}$ according to 
\begin{equation}
\widehat{I^\sigma_N v}(\xi):=m^{\sigma}\left(\frac{\xi}{N}\right)\widehat{v}(\xi).
\label{IsN-def}
\end{equation}
It is straightforward to verify that $I^\sigma_N$ is a smoothing operator of order $\sigma$, in the sense that
\begin{equation}
\|v\|_{H^{\tilde \sigma}}\lesssim \|I^\sigma_N v\|_{H^{\tilde{\sigma}+\sigma}}\lesssim N^\sigma\|v\|_{H^{\tilde \sigma}}.
\label{I-smooth}
\end{equation}
Next, we recall an interpolation result (Lemma 12.1 in \cite{CKSTT-04}) which yields multilinear estimates related to this family of operators.

\begin{lemma}
Let $\sigma_0>0$ and $n\geq 1$. Suppose that $Z$, $X_1,\ldots, X_n$ are translation invariant Banach spaces and T is a translation invariant n-linear operator such that one has the estimate
\[
\|I_1^\sigma T(u_1, \ldots, u_n)\|_Z\lesssim \prod_{i=1}^n \|I_1^\sigma u_i\|_{X_i}
\]
for all $u_1,\ldots, u_n$ and all $0\leq \sigma\leq\sigma_0$. Then one has the estimate
\[
\|I_N^\sigma T(u_1, \ldots, u_n)\|_Z\lesssim \prod_{i=1}^n \|I_N^\sigma u_i\|_{X_i}
\]
for all $u_1,\ldots, u_n$, all $0\leq \sigma\leq\sigma_0$, and $N\geq 1$, with the implicit constant independent of $N$. 
\end{lemma}

In what follows, we will be mainly working with $I_N^{2-s}$, where $s\leq 2$. Our goal is to prove
\begin{equation}
\|I_N^{2-s}(|w|^{2k}w)\|_{X^{1,0}}\lesssim \|I_N^{2-s}w\|^{2k+1}_{X^{2,\frac{1}{2}+}}, \qquad k\geq 2,\quad \frac{1}{2}-\frac{2}{2k+1}\leq s \leq 2,
\label{multi-I}
\end{equation}
an important bound to be relied on in the next section. This is the consequence of the following multilinear estimate.

\begin{lemma} 
Let $k\geq 2$ and $1/2-2/(2k+1)\leq s \leq 2$. Then
\begin{equation}
\|I_N^{2-s}(w_1\cdot\ldots \cdot w_{2k+1})\|_{X^{1,0}}\lesssim \|I_N^{2-s}w_1\|_{X^{2,\frac{1}{2}+}}\ldots\|I_N^{2-s}w_{2k+1}\|_{X^{2,\frac{1}{2}+}}
\label{multilin-I}
\end{equation}
holds true.
\end{lemma}
\begin{proof}
According to the interpolation lemma, the claim is valid if we prove
\begin{equation*}
\|I_1^{2-s}(w_1\cdot\ldots \cdot w_{2k+1})\|_{X^{1,0}}\lesssim \|I_1^{2-s}w_1\|_{X^{2,\frac{1}{2}+}}\ldots\|I_1^{2-s}w_{2k+1}\|_{X^{2,\frac{1}{2}+}}
\end{equation*}
under the same restrictions for $k$ and $s$. However, based on the definition of $m$, we can work with $I_1^{2-s}\simeq J^{s-2}$ and, hence, the previous bound can be restated as 
\begin{equation*}
\|J^{s-2}(w_1\cdot\ldots \cdot w_{2k+1})\|_{X^{1,0}}\lesssim \|J^{s-2}w_1\|_{X^{2,\frac{1}{2}+}}\ldots\|J^{s-2}w_{2k+1}\|_{X^{2,\frac{1}{2}+}}.
\end{equation*}
Since $\|J^{s-2}w\|_{X^{\sigma,\theta}}=\|w\|_{X^{s+\sigma-2,\theta}}$, this translates into
\[
\|w_1\cdot\ldots \cdot w_{2k+1}\|_{X^{s-1,0}}\lesssim \|w_1\|_{X^{s,\frac{1}{2}+}}\ldots\|w_{2k+1}\|_{X^{s,\frac{1}{2}+}},
\]
which is the estimate \eqref{prod-xst} proven before.
\end{proof}


\section{Adapted local well-posedness theory} \label{LWP-th}

The fundamental idea behind the \textit{I-method} is that it treats equations having rough data by means of similar equations with smoothed out data, which are obtained, in turn, with the help of the multiplier operators introduced before. Precisely, due to \eqref{I-smooth}, we know that if $u$ solves the IVP \eqref{main} on the time interval $[0,\delta]$ with $(g,h)\in H^s(\R)\times H^{s-2}(\R)$  and $s<2$, then $I_N^{2-s}$, which is renamed onward $Iu$ to simplify notation, solves\footnote{Given that $I$ is a multiplier operator, it commutes with any derivative, either in $t$ or in $x$.}
\begin{equation}
\begin{cases}
(Iu)_{tt}-(Iu)_{xx}-\beta (Iu)_{xxxx}-(Iu)_{xxxxxx}\,=\,(I(f(u)))_{xx}, \\
Iu(0,x)\,=\,Ig(x),\qquad (Iu)_t(0,x)\,=\,(Ih)_x(x),\\
\end{cases}
\label{main-I}
\end{equation}
on the same time interval with $(Ig,Ih)\in H^2(\R)\times L^{2}(\R)$ and vice versa. The tools developed in the previous section allow us to obtain a LWP result for the smoothed out IVP. As mentioned in the introduction, the absence of scaling invariance for generalized sixth-order Boussinesq equations creates the extra task of deriving independently asymptotics on the size of the interval of existence associated to \eqref{main-I}.

\begin{theorem}\label{LWP-I}
Assume that $k\geq 2$ is an integer and $(g,h)\in H^s(\R)\times H^{s-2}(\R)$ with $1/2-2/(2k+1)\leq s<2$. There exists $0<\delta<1$ such that the IVP \eqref{main-I} with $f(u)=|u|^{2k}u$ admits a unique solution $Iu\in C([0,\delta]; H^2(\R))$ satisfying 
\begin{equation}
\|Iu\|_{X^{2,\frac{1}{2}+}_\delta}+\|(-\Delta)^{-1/2}Iu_t\|_{X^{0,\frac{1}{2}+}_\delta}\lesssim \|Ig\|_{H^2(\R)}+\|Ih\|_{L^2(\R)}
\label{Id-est}
\end{equation}
and 
\begin{equation}
\delta^{\frac 12 -}\lesssim \frac{1}{(\|Ig\|_{H^2(\R)}+\|Ih\|_{L^2(\R)})^{2k}}.
\label{d-est}
\end{equation}
In particular, the maximal time of existence can be approximated by writing $\simeq$ in place of $\lesssim$ in the previous estimate. 
\end{theorem}

\begin{proof}
We demonstrate the result by using a fixed-point argument for the equation
\begin{equation*}
w=\eta w^{(1)}+\eta_\delta w^{(2)},
\end{equation*}
where $0<\delta<1$,
\begin{equation*}
w^{(1)}_{tt}-w^{(1)}_{xx}-\beta w^{(1)}_{xxxx}-w^{(1)}_{xxxxxx}\,=\,0, \quad w^{(1)}(0)\,=\,Ig,\quad w^{(1)}_t(0)\,=\,(Ih)_x,
\end{equation*}
and
\begin{equation*}
w^{(2)}_{tt}-w^{(2)}_{xx}-\beta w^{(2)}_{xxxx}-w^{(2)}_{xxxxxx}\,=\,(I(|I^{-1}w|^{2k}I^{-1}w))_{xx}, \quad w^{(2)}(0)\,=\,w^{(2)}_t(0)\,=\,0.
\end{equation*}
If we denote the right-hand side of the above equation by $T(w)$, the goal is to show that $T$ is a contraction on a closed ball of the Banach space 
\[
\aligned
&W:=\{w \in X^{2,\frac{1}{2}+}, (-\Delta)^{-1/2}w_t\in X^{0,\frac{1}{2}+}\},\\ &\|w\|_W:=\|w\|_{X^{2,\frac{1}{2}+}}+\|(-\Delta)^{-1/2}w_t\|_{X^{0,\frac{1}{2}+}}.
\endaligned
\] 

We proceed by relying on \eqref{lin-hom} with $(\sigma,\theta)=(2, 1/2+)$, \eqref{lin-inhom} with $(\sigma,\theta_1,\theta_2)=(2, 1/2+, 0)$, and \eqref{multi-I}-\eqref{multilin-I} with $k$ and $s$ like in the statement of the theorem to infer
\begin{equation*}
\aligned
\|T(w)\|_W&\lesssim \|Ig\|_{H^2(\R)}+\|Ih\|_{L^2(\R)}+\delta^{\frac{1}{2}-}\|(-\Delta)^{-1/2}(I(|I^{-1}w|^{2k}I^{-1}w))_{xx}\|_{X^{0,0}}\\
&\lesssim \|Ig\|_{H^2(\R)}+\|Ih\|_{L^2(\R)}+\delta^{\frac{1}{2}-}\|I(|I^{-1}w|^{2k}I^{-1}w))\|_{X^{1,0}}\\
&\lesssim \|Ig\|_{H^2(\R)}+\|Ih\|_{L^2(\R)}+\delta^{\frac{1}{2}-}\|w\|^{2k+1}_{X^{2,\frac{1}{2}+}}\\
&\lesssim \|Ig\|_{H^2(\R)}+\|Ih\|_{L^2(\R)}+\delta^{\frac{1}{2}-}\|w\|^{2k+1}_W
\endaligned
\end{equation*}
and
\begin{equation*}
\aligned
\|T(w)-T(\tilde{w})\|_W&\lesssim \delta^{\frac{1}{2}-}\|(-\Delta)^{-1/2}(I(|I^{-1}w|^{2k}I^{-1}w-|I^{-1}\tilde{w}|^{2k}I^{-1}\tilde{w})_{xx}\|_{X^{0,0}}\\
&\lesssim \delta^{\frac{1}{2}-}\|I(|I^{-1}w|^{2k}I^{-1}w-|I^{-1}\tilde{w}|^{2k}I^{-1}\tilde{w})\|_{X^{1,0}}\\
&\lesssim \delta^{\frac{1}{2}-}(\|w\|^{2k}_{X^{2,\frac{1}{2}+}}+\|\tilde w\|^{2k}_{X^{2,\frac{1}{2}+}})\|w-\tilde{w}\|_{X^{2,\frac{1}{2}+}}\\
&\lesssim \delta^{\frac{1}{2}-}(\|w\|^{2k}_{W}+\|\tilde w\|^{2k}_{W})\|w-\tilde{w}\|_W.
\endaligned
\end{equation*}
Hence, by choosing 
\[
\delta^{\frac 12 -}(\|Ig\|_{H^2(\R)}+\|Ih\|_{L^2(\R)})^{2k}\lesssim 1,
\]
we deduce that $T$ is a contraction on a closed ball centered at the origin in $W$, whose radius $R$ satisfies
\[
R\simeq \|Ig\|_{H^2(\R)}+\|Ih\|_{L^2(\R)}.
\] 
It follows that the fixed point $w=Iu$ of the map $T$ is a solution to the IVP \eqref{main-I} on the time interval $[0,\delta]$, which also leads to \eqref{Id-est}. Using \eqref{Sob-X}, we obtain $w\in C(\R; H^2(\R))$ and, hence, $Iu\in C([0,\delta]; H^2(\R))$.
\end{proof}


\section{Key multilinear estimate}\label{multilinear}

In this section, for ease of notation, we write
\begin{equation*}
\int_{\xi_1+\ldots+\xi_n=0} f(\xi_1, \ldots, \xi_n)=\int_{\R^{n-1}} f(-\xi_2-\ldots-\xi_n, \xi_2, \ldots, \xi_n)\,d\xi_2\ldots d\xi_n
\end{equation*}
and we label by $M=M(\xi)$ the Fourier multiplier associated to the multiplier operator $I$, which is given according to \eqref{m-def} and \eqref{IsN-def} by
\begin{equation*}
\widehat{Iv}(\xi)=\widehat{I^{2-s}_N v}(\xi)=m^{2-s}\left(\frac{\xi}{N}\right)\widehat{v}(\xi).
\end{equation*}
Besides the previous LWP result, another crucial ingredient for proving Theorem \ref{main-th} is the following multilinear estimate.

\begin{theorem}
Let $k\geq 2$ be an integer, $\delta>0$,
\[
s>
\begin{cases}
\frac{1}{4}, \quad \text{if} \ \ k=2,\\
\frac{1}{2}, \quad \text{if}  \  \ k>2,\\
\end{cases}
\]
and take $N\geq 1$ to be sufficiently large depending on $s$. Under these assumptions,
\begin{equation}
\aligned
\Bigg | \int_0^\delta \int_{\xi_1+\ldots+\xi_{2k+2}=0} & \left(1-\frac{M(\xi_2+\ldots+\xi_{2k+2})}{M(\xi_2)\ldots M(\xi_{2k+2})}\right)\\
&\cdot |\xi_1|\widehat{Iw_1}(t,
\xi_1)\widehat{Iw_2}(t,\xi_2)\ldots\widehat{Iw_{2k+2}}(t,\xi_{2k+2})\,dt\Bigg|\\
\lesssim N^{-4+}&\|Iw_1\|_{X^{0, \frac{1}{2}+}_\delta}\|Iw_2\|_{X^{2, \frac{1}{2}+}_\delta}\ldots\|Iw_{2k+2}\|_{X^{2, \frac{1}{2}+}_\delta}
\endaligned
\label{multi-I-int}
\end{equation}
holds true.
\end{theorem}
\begin{proof}
In arguing for the above bound, we first make the observation that it is the consequence of the slightly sharper dyadic version, i.e., 
\begin{equation}
\aligned
\Bigg | \int_0^\delta \int\limits_{\underset{|\xi_i|\simeq N_i\in 2^{\mathbb{Z}}}{\xi_1+\ldots+\xi_{2k+2}=0}} & \left(1-\frac{M(\xi_2+\ldots+\xi_{2k+2})}{M(\xi_2)\ldots M(\xi_{2k+2})}\right)\\
&\cdot |\xi_1|\widehat{Iw_1}(t,
\xi_1)\widehat{Iw_2}(t,\xi_2)\ldots\widehat{Iw_{2k+2}}(t,\xi_{2k+2})\,dt\Bigg|\\
\lesssim N^{-4+}N^{0-}_{\text{max}}&\|Iw_1\|_{X^{0, \frac{1}{2}+}_\delta}\|Iw_2\|_{X^{2, \frac{1}{2}+}_\delta}\ldots\|Iw_{2k+2}\|_{X^{2, \frac{1}{2}+}_\delta},
\endaligned
\label{multi-I-int-dyadic}
\end{equation}
where $\displaystyle N_{\text{max}}=\max_{1\leq i\leq 2k+2}N_i$. 

Next, based on the symmetry of this estimate in the $(\xi_2,\ldots,\xi_{2k+2})$ variables, we can make the assumption to only work in the $N_2\geq\ldots\geq N_{2k+2}$ regime. Moreover, another simplifying reduction is attained by noticing that on the domain of integration we have
\[
|\xi_1|\leq |\xi_2|+\ldots+|\xi_{2k+2}|,
\] 
which implies $N_1\lesssim N_2$. Finally, we can also assume that $N_2\gtrsim N$, since $N_2\ll N$ would lead to $M(\xi_1)=\ldots=M(\xi_{2k+2})=1$ and, thus,
\[
1-\frac{M(\xi_1)}{M(\xi_2)\ldots M(\xi_{2k+2})}=0.
\] 
Before starting the actual proof of \eqref{multi-I-int-dyadic}, we make one more notational convention to actually write $M(N_i)$ for $M(\xi_i)$, given the definition of $M$ and the dyadic localization of $|\xi_i|$. Additionally, we claim that a calculus-level analysis allows us to work, for all intended purposes, with $x\mapsto \langle x\rangle^\alpha M(x)$ being nondecreasing on $\R_+$ if $\alpha+s>2$ and $N$ is sufficiently large depending on $\alpha$ and $s$.

The argument consists in analyzing separately the complementary cases $N_2\gtrsim N\gg N_3$, $N_2\gg N_3\gtrsim N$, and $N_2\simeq N_3\gtrsim N$. For the first one, since $\xi_1+\ldots+\xi_{2k+2}=0$, we have
\[
N_1\simeq N_2 \gtrsim N\gg N_3\geq\ldots\geq N_{2k+2},
\]
which further implies
\[
M(\xi_3)=\ldots= M(\xi_{2k+2})=1.
\] 
It follows that
\begin{equation*}
\aligned
\bigg|1-&\frac{M(\xi_2+\ldots+\xi_{2k+2})}{M(\xi_2)\ldots M(\xi_{2k+2})}\bigg|= \left|1-\frac{M(\xi_2+\ldots+\xi_{2k+2})}{M(\xi_2)}\right|\\
&\qquad\leq \frac{\sup_{0\leq a\leq 1}|M'(\xi_2+a(\xi_3+\ldots+\xi_{2k+2}))|\cdot |\xi_3+\ldots+\xi_{2k+2}|}{M(\xi_2)}\\
&\qquad\lesssim \frac{|M'(N_2)|N_3}{M(N_2)}\lesssim \frac{N_3}{N_2},
\endaligned
\end{equation*}
with the last estimate being the consequence of $N_2\gtrsim N$ and of the definition of $M$. Now, we also take advantage of \eqref{w-K}, \eqref{w-mf}, and \eqref{w-infty} (only if $k>2$) to deduce
\begin{equation*}
\aligned
\text{(LHS) of \eqref{multi-I-int-dyadic}}\ \lesssim\ &\frac{N_1N_3}{N_2}\|Iw_1\ldots Iw_{2k+2}\|_{L^1_{x,t\in[0,\delta]}}\\
\lesssim\ &N_3\|Iw_1\|_{L^\infty_xL^2_{t\in[0,\delta]}}\|Iw_2\|_{L^\infty_xL^2_{t\in[0,\delta]}}\|Iw_3\|_{L^4_xL^\infty_{t\in[0,\delta]}}\ldots \|Iw_6\|_{L^4_xL^\infty_{t\in[0,\delta]}}\\
&\cdot \|Iw_{7}\|_{L^\infty_{x,t\in[0,\delta]}}\ldots \|Iw_{2k+2}\|_{L^\infty_{x,t\in[0,\delta]}}\\
\lesssim\ &N_3\|Iw_1\|_{X^{-1, \frac{1}{2}+}_\delta}\|Iw_2\|_{X^{-1, \frac{1}{2}+}_\delta}\|Iw_3\|_{X^{\frac{1}{4}, \frac{1}{2}+}_\delta}\ldots \|Iw_6\|_{X^{\frac{1}{4}, \frac{1}{2}+}_\delta}\\
&\cdot \|Iw_{7}\|_{X^{\frac{1}{2}+, \frac{1}{2}+}_\delta}\ldots \|Iw_{2k+2}\|_{X^{\frac{1}{2}+, \frac{1}{2}+}_\delta}\\
\lesssim\ &\frac{N_3}{\langle N_1\rangle\langle N_2\rangle^3\langle N_3\rangle^{\frac{7}{4}}\ldots\langle N_6\rangle^{\frac{7}{4}}\langle N_7\rangle^{\frac{3}{2}-}\ldots\langle N_{2k+2}\rangle^{\frac{3}{2}-}}\\
&\cdot\|Iw_1\|_{X^{0, \frac{1}{2}+}_\delta}\|Iw_2\|_{X^{2, \frac{1}{2}+}_\delta}\ldots \|Iw_{2k+2}\|_{X^{2, \frac{1}{2}+}_\delta}\\
\lesssim\ &\frac{1}{\langle N_2\rangle^4}\|Iw_1\|_{X^{0, \frac{1}{2}+}_\delta}\|Iw_2\|_{X^{2, \frac{1}{2}+}_\delta}\ldots \|Iw_{2k+2}\|_{X^{2, \frac{1}{2}+}_\delta}\\
\lesssim\ &\frac{1}{N^{4-}N_2^{0+}}\|Iw_1\|_{X^{0, \frac{1}{2}+}_\delta}\|Iw_2\|_{X^{2, \frac{1}{2}+}_\delta}\ldots \|Iw_{2k+2}\|_{X^{2, \frac{1}{2}+}_\delta},
\endaligned
\end{equation*}
which proves the claim in this case. 

For the remaining two scenarios, due to $N_1\lesssim N_2$, $M$ being even, nonincreasing on $\R_+$, and $0<M\leq 1$, we can estimate the symbol in the integral as
\begin{equation*}
\aligned
\left|1-\frac{M(\xi_2+\ldots+\xi_{2k+2})}{M(\xi_2)\ldots M(\xi_{2k+2})}\right|&= \left|1-\frac{M(\xi_1)}{M(\xi_2)\ldots M(\xi_{2k+2})}\right|\\&\lesssim 1+ \frac{M(N_1)}{M(N_2)\ldots M(N_{2k+2})}\\
&\lesssim \frac{M(N_1)}{M(N_2)\ldots M(N_{2k+2})}.
\endaligned
\end{equation*}
If we are in the case $N_2\gg N_3\gtrsim N$, then we have, as in the first one, $N_1\simeq N_2$ and we can similarly derive
\begin{equation*}
\aligned
\text{(LHS) of \eqref{multi-I-int-dyadic}}\ \lesssim\ &\frac{N_1M(N_1)}{M(N_2)\ldots M(N_{2k+2})}\|Iw_1\ldots Iw_{2k+2}\|_{L^1_{x,t\in[0,\delta]}}\\
\lesssim\ &\frac{N_2}{M(N_3)\ldots M(N_{2k+2})}\|Iw_1\|_{X^{-1, \frac{1}{2}+}_\delta}\|Iw_2\|_{X^{-1, \frac{1}{2}+}_\delta}\\
&\cdot \|Iw_3\|_{X^{\frac{1}{4}, \frac{1}{2}+}_\delta}\ldots \|Iw_6\|_{X^{\frac{1}{4}, \frac{1}{2}+}_\delta}\|Iw_{7}\|_{X^{\frac{1}{2}+, \frac{1}{2}+}_\delta}\ldots \|Iw_{2k+2}\|_{X^{\frac{1}{2}+, \frac{1}{2}+}_\delta}\\
\lesssim\ &\frac{N_2}{\langle N_1\rangle\langle N_2\rangle^3}\frac{1}{\langle N_3\rangle^{\frac{7}{4}}M(N_3)\ldots\langle N_6\rangle^{\frac{7}{4}}M(N_6)}\\
&\cdot \frac{1}{\langle N_7\rangle^{\frac{3}{2}-}M(N_7)\ldots\langle N_{2k+2}\rangle^{\frac{3}{2}-}M(N_{2k+2})}\\
&\cdot\|Iw_1\|_{X^{0, \frac{1}{2}+}_\delta}\|Iw_2\|_{X^{2, \frac{1}{2}+}_\delta}\ldots \|Iw_{2k+2}\|_{X^{2, \frac{1}{2}+}_\delta}.\endaligned
\end{equation*}
As argued before, since $s>1/4$ and $N$ is sufficiently large depending on $s$, we can rely on  $x\mapsto \langle x\rangle^{7/4} M(x)$ being nondecreasing on $\R_+$ and, thus, we have
\begin{equation*}
\langle N_3\rangle^{\frac{7}{4}}M(N_3)\gtrsim \langle N\rangle^{\frac{7}{4}}M(N)=\langle N\rangle^{\frac{7}{4}}, \qquad \langle N_i\rangle^{\frac{7}{4}}M(N_i)\geq 1, \quad 4\leq i\leq 6.
\end{equation*}
If $k>2$, one needs to also use $s>1/2$ to infer
\begin{equation*}
\langle N_i\rangle^{\frac{3}{2}-}M(N_i)\geq 1, \quad 7\leq i\leq 2k+2.
\end{equation*}
Hence, we can follow up in estimating the integral and deduce
\begin{equation*}
\aligned
\text{(LHS) of \eqref{multi-I-int-dyadic}}\ 
\lesssim\ &\frac{1}{\langle N_2\rangle^3\langle N\rangle^{\frac{7}{4}}}\|Iw_1\|_{X^{0, \frac{1}{2}+}_\delta}\|Iw_2\|_{X^{2, \frac{1}{2}+}_\delta}\ldots \|Iw_{2k+2}\|_{X^{2, \frac{1}{2}+}_\delta}\\
\lesssim\ &\frac{1}{N^{\frac{19}{4}-}N_2^{0+}}\|Iw_1\|_{X^{0, \frac{1}{2}+}_\delta}\|Iw_2\|_{X^{2, \frac{1}{2}+}_\delta}\ldots \|Iw_{2k+2}\|_{X^{2, \frac{1}{2}+}_\delta},
\endaligned
\end{equation*}
which is an even sharper bound than the one obtained in the first case.

Finally, when $N_2\simeq N_3\gtrsim N$, the analysis changes slightly from the one in the second case, with $Iw_1$ and $Iw_3$ being estimated now in $L^4_xL^\infty_{t\in[0,\delta]}$ and $L^\infty_xL^2_{t\in[0,\delta]}$, respectively. Accordingly, we obtain
\begin{equation*}
\aligned
\text{(LHS) of \eqref{multi-I-int-dyadic}}\ \lesssim\ &\frac{N_1M(N_1)}{M(N_2)\ldots M(N_{2k+2})}\|Iw_1\ldots Iw_{2k+2}\|_{L^1_{x,t\in[0,\delta]}}\\
\lesssim\ &\frac{N_1M(N_1)}{M(N_2)\ldots M(N_{2k+2})}\|Iw_1\|_{L^4_xL^\infty_{t\in[0,\delta]}}\|Iw_2\|_{L^\infty_xL^2_{t\in[0,\delta]}}\\
&\cdot\|Iw_3\|_{L^\infty_xL^2_{t\in[0,\delta]}} \|Iw_4\|_{L^4_xL^\infty_{t\in[0,\delta]}}\|Iw_5\|_{L^4_xL^\infty_{t\in[0,\delta]}} \|Iw_6\|_{L^4_xL^\infty_{t\in[0,\delta]}}\\
&\cdot\|Iw_{7}\|_{L^\infty_{x,t\in[0,\delta]}}\ldots \|Iw_{2k+2}\|_{L^\infty_{x,t\in[0,\delta]}}\\
\lesssim\ &\frac{N^{\frac{5}{4}}_1M(N_1)}{M(N_2)\ldots M(N_{2k+2})}\|Iw_1\|_{X^{0, \frac{1}{2}+}_\delta}\|Iw_2\|_{X^{-1, \frac{1}{2}+}_\delta}\\
&\cdot \|Iw_3\|_{X^{-1, \frac{1}{2}+}_\delta} \|Iw_4\|_{X^{\frac{1}{4}, \frac{1}{2}+}_\delta}\|Iw_5\|_{X^{\frac{1}{4}, \frac{1}{2}+}_\delta} \|Iw_6\|_{X^{\frac{1}{4}, \frac{1}{2}+}_\delta}\\
&\cdot \|Iw_{7}\|_{X^{\frac{1}{2}+, \frac{1}{2}+}_\delta}\ldots \|Iw_{2k+2}\|_{X^{\frac{1}{2}+, \frac{1}{2}+}_\delta}\\
\lesssim\ &\frac{N^{\frac{5}{4}}_1M(N_1)}{\langle N_2\rangle^3M(N_2)\langle N_3\rangle^3M(N_3)}\\
&\cdot\frac{1}{\langle N_4\rangle^{\frac{7}{4}}M(N_4)\langle N_5\rangle^{\frac{7}{4}}M(N_5)\langle N_6\rangle^{\frac{7}{4}}M(N_6)}\\
&\cdot \frac{1}{\langle N_7\rangle^{\frac{3}{2}-}M(N_7)\ldots\langle N_{2k+2}\rangle^{\frac{3}{2}-}M(N_{2k+2})}\\
&\cdot\|Iw_1\|_{X^{0, \frac{1}{2}+}_\delta}\|Iw_2\|_{X^{2, \frac{1}{2}+}_\delta}\ldots \|Iw_{2k+2}\|_{X^{2, \frac{1}{2}+}_\delta}\\
\lesssim\ &\frac{N^{\frac{5}{4}}_1M(N_1)}{\langle N_2\rangle^6M^2(N_2)}\|Iw_1\|_{X^{0, \frac{1}{2}+}_\delta}\|Iw_2\|_{X^{2, \frac{1}{2}+}_\delta}\ldots \|Iw_{2k+2}\|_{X^{2, \frac{1}{2}+}_\delta}\\
\lesssim\ &\frac{1}{N^{\frac 72}N_2^{\frac 54}}\|Iw_1\|_{X^{0, \frac{1}{2}+}_\delta}\|Iw_2\|_{X^{2, \frac{1}{2}+}_\delta}\ldots \|Iw_{2k+2}\|_{X^{2, \frac{1}{2}+}_\delta}\\
\lesssim\ &\frac{1}{N^{\frac{19}{4}-}N_2^{0+}}\|Iw_1\|_{X^{0, \frac{1}{2}+}_\delta}\|Iw_2\|_{X^{2, \frac{1}{2}+}_\delta}\ldots \|Iw_{2k+2}\|_{X^{2, \frac{1}{2}+}_\delta},
\endaligned
\end{equation*}
where the line before the last one is due to $N_1\lesssim N_2$, $M\leq 1$, and
\[
\langle N_2\rangle^{\frac{7}{4}}M(N_2)\gtrsim \langle N\rangle^{\frac{7}{4}}.
\]
This finishes the argument for the theorem.
\end{proof}

\begin{remark}
As commented in Subsection \ref{main-result}, this multilinear estimate is sharper than its $k=1$ counterpart in \cite{WE-14}, with $N^{-4+}$ replacing $N^{-3+}$. Hence, we are able to prove GWP for $2-2/(3k)<s<2$ rather than for the expected $2-1/(2k)<s<2$ range. Furthermore, our proof of \eqref{multi-I-int} does not require splitting the discussion of the $N_2\simeq N_3\gtrsim N$ scenario into four separate subcases as in \cite{WE-14}.  
\end{remark}


\section{Proof of the main result}

Now, we have all the elements in place to prove Theorem \ref{main-th}. The strategy, much like with other applications of the \textit{I-method} is to use iteratively the adapted LWP result (i.e., Theorem \ref{LWP-I}) in order to reach an arbitrary time of existence $T>0$ for the solution $u$ to \eqref{main}. According to \eqref{d-est}, this is possible if we control the growth of 
\[
t\mapsto \|Iu(t)\|_{H^2(\R)}+\|(-\Delta)^{-1/2}Iu_t(t)\|_{L^{2}(\R)}.
\]
We achieve this by using, among others, the energy \eqref{energy} and the multilinear estimate \eqref{multi-I-int}.

\begin{proof}[Proof of Theorem \ref{main-th}]
We start by invoking \eqref{I-smooth} to claim
\begin{equation}
\|Ig\|_{H^2(\R)}+\|Ih\|_{L^2(\R)}\lesssim N^{2-s}(\|g\|_{H^s(\R)}+\|h\|_{H^{s-2}(\R)})\lesssim N^{2-s}.
\label{Ig-N}
\end{equation}
If we couple this bound with an application of Theorem \ref{LWP-I} (in particular \eqref{Id-est}), we derive that a solution to \eqref{main-I} with $f(u)=|u|^{2k}u$ satisfies
\begin{equation}
\|Iu\|_{X^{2,\frac{1}{2}+}_\delta}+\|(-\Delta)^{-1/2}Iu_t\|_{X^{0,\frac{1}{2}+}_\delta}\lesssim N^{2-s},
\label{Id-N}
\end{equation}
with 
\begin{equation}
\delta^{\frac{1}{2}-}\simeq N^{-2k(2-s)}.
\label{d-N}
\end{equation}

Next, we follow the standard procedure of obtaining energy estimates for $Iu$ (i.e., we apply the multiplier operator $(-\Delta)^{-1/2}$ to \eqref{main-I}, multiply the resulting equation by $(-\Delta)^{-1/2}Iu_t$, and integrate by parts with respect to the spatial variable), which implies
\begin{equation*}
\aligned
\frac{1}{2}\frac{d}{dt}\Big\{&\|(Iu)_{xx}(t)\|^2_{L^2(\R)}-\beta\|(Iu)_x(t)\|^2_{L^2(\R)}+\|Iu(t)\|^2_{L^2(\R)}\\
&+\|(-\Delta)^{-1/2}(Iu)_t(t)\|^2_{L^{2}(\R)}\Big\}+\int_\R I(|u|^{2k}u)(t,x)Iu_t(t,x)\,dx=0.
\endaligned
\end{equation*}
Taking into account now \eqref{energy}, we infer
\begin{equation*}
\frac{d}{dt}\{E(Iu)(t)\}= \int_\R \left(|Iu(t,x)|^{2k}Iu(t,x)- I(|u|^{2k}u)(t,x)\right)Iu_t(t,x)\,dx.
\end{equation*}
If we factor in the fundamental theorem of calculus and Parseval's formula, then we deduce
\begin{equation*}
\aligned
E(Iu)(\delta)&-E(Iu)(0)\\
&=\int_0^\delta\int_\R \left(|Iu(t,x)|^{2k}Iu(t,x)- I(|u|^{2k}u)(t,x)\right)Iu_t(t,x)\,dx\,dt\\
&=\int_0^\delta \int\limits_{\xi_1+\ldots+\xi_{2k+2}=0} \bigg\{ \left(1-\frac{M(\xi_2+\ldots+\xi_{2k+2})}{M(\xi_2)\ldots M(\xi_{2k+2})}\right)\\
&\qquad\qquad\qquad\qquad\quad\cdot\widehat{Iu_t}(t,
\xi_1)\widehat{Iu}(t,\xi_2)\ldots\widehat{Iu}(t,\xi_{2k+2})\bigg\}\,dt.
\endaligned
\end{equation*}
This is the point in the argument where we use the multilinear estimate \eqref{multi-I-int} and \eqref{Id-N} to derive
\begin{equation*}
\aligned
|E(Iu)(\delta)-E(Iu)(0)|\lesssim N^{-4+} \|(-\Delta)^{-1/2}Iu_t\|_{X^{0,\frac{1}{2}+}_\delta}\|Iu\|^{2k+1}_{X^{2,\frac{1}{2}+}_\delta}\lesssim N^{(2k+2)(2-s)-4+}.
\endaligned
\end{equation*}
We also note that, based on \eqref{energy-bd}, \eqref{Ig-N}, Sobolev embeddings, and Mikhlin's multiplier theorem, we have
\begin{equation*}
\aligned
E(Iu)(0)&\simeq \|Ig\|^2_{H^2(\R)}+\|Ih\|^2_{L^{2}(\R)}+\|Ig\|^{2k+2}_{L^{2k+2}(\R)}\lesssim N^{4-2s}+\|g\|^{2k+2}_{L^{2k+2}(\R)}\\
&\lesssim N^{4-2s}+\|g\|^{2k+2}_{H^{\frac 12}(\R)}\lesssim N^{4-2s}.
\endaligned
\end{equation*}
If $N^{(2k+2)(2-s)-4+}\ll N^{4-2s}$, the last two inequalities imply $E(Iu)(\delta)\lesssim N^{4-2s}$. Due to \eqref{energy-bd}, we obtain
\begin{equation*}
\|Iu(\delta)\|_{H^{2}(\R)}+\|(-\Delta)^{-1/2}(Iu)_t(\delta)\|_{L^{2}(\R)}\lesssim N^{2-s}
\end{equation*}
and, consequently, we can run one more time what we have done so far, now on the time interval $[\delta, 2\delta]$. 

It follows that for a fixed time $T>0$, we can perform $T/\delta$ iterations of the previous scheme to cover $[0,T]$ if the energy $E(Iu)(t)$ doesn't double in size on this interval. This happens if
\[
\frac{N^{(2k+2)(2-s)-4+}T}{\delta}\ll N^{4-2s}
\]
holds true and, taking into account \eqref{d-N}, we can ensure this is the case if 
\[
T\simeq N^{6ks-12k+4-}.
\]
Since $2-2/(3k)<s$, the exponent of $N$ is positive and, thus, arbitrary large times of existence $T$ can be reached by choosing $N\gg1$ appropriately. 

Finally, using \eqref{I-smooth} and \eqref{energy-bd}, we infer
\begin{equation*}
\aligned
\sup_{0\leq t\leq T}& \left\{\|u(t)\|^2_{H^s(\R)}+\|(-\Delta)^{-1/2}u_t(t)\|^2_{H^{s-2}(\R)}\right\}\\
&\lesssim\sup_{0\leq t\leq T} \left\{\|Iu(t)\|^2_{H^2(\R)}+\|(-\Delta)^{-1/2}(Iu)_t(t)\|^2_{L^{2}(\R)}\right\}\\
&\lesssim\sup_{0\leq t\leq T} E(Iu)(t)\lesssim N^{4-2s}\simeq T^{\frac{4-2s}{6ks-12k+4}+}, 
\endaligned
\end{equation*}
which proves \eqref{u-hs} and finishes the whole argument.
\end{proof}

\bibliographystyle{amsplain}
\bibliography{bousbib}

\end{document}